\documentclass{amsart}
\usepackage{amssymb,latexsym,amsmath,amsthm, graphicx, epsfig}
\usepackage{youngtab}

\newtheorem{theorem}{Theorem}
\newtheorem{corollary}{Corollary}
\newtheorem{lemma}{Lemma}
\newtheorem{definition}{Definition}

\newtheorem{conjecture}{Conjecture}

\begin{document}
\title[Partitions into two sizes of part]{Combinatorial proof of a congruence for partitions into two sizes of part}
\author{Eli R. DeWitt and William J. Keith} \address{Department of Mathematical Sciences\\ Michigan Tech\\ Houghton, MI  49931-1295}
\email{erdewitt@mtu.edu, wjkeith@mtu.edu}

\begin{abstract}
Previous work showed that, for $\nu_2(n)$ the number of partitions of $n$ into exactly two part sizes, one has $\nu_2(16n+14) \equiv 0 \pmod{4}$.  The earlier proof required the technology of modular forms, and a combinatorial proof was desired.  This article provides the requested proof, in the process refining divisibility to finer subclasses.  Some of these subclasses have counts closely related to the divisor function $d(16n+14)$, and we offer a conjecture on a potential rank statistic.
\end{abstract}

\maketitle

\section{Introduction}

Let the number of partitions of $n$ in which exactly $k$ sizes of part appear be denoted by $\nu_k(n)$.  For example, $\nu_2(6) = 6$, counting $$(5,1), (4,2), (4,1,1), (3,1,1,1), (2,2,1,1), (2,1,1,1,1).$$  Of the remaining five partitions of 6, $\nu_3(6) = 1$ counts $(3,2,1)$ and $\nu_1(6) = 4$ counts the partitions $(6)$, $(3,3)$, $(2,2,2)$, $(1,1,1,1,1,1)$.  Clearly $\nu_1(n) = d(n)$, the number of divisors of $n$, which is perfectly understood, and so combinatorial interest is greater for $\nu_k(n)$ when $k>1$.

These functions have been studied by Major P. A. MacMahon \cite{MacMahon}, George Andrews \cite{GEA1}, and Tani and Bouroubi \cite{TandB}, the latter specifically interested in $\nu_2$.  Their generating functions are $$\sum_{n=0}^\infty \nu_k(n) q^n = \sum_{1 \leq a_1 < a_2 \dots \leq a_k} \frac{q^{a_1 + \dots + a_k}}{(1-q^{a_1}) \dots (1-q^{a_k})}.$$  These functions are close cousins of a class of functions, defined in the same paper of MacMahon, that are an area of much recent activity: the quasi-modular forms given by the product-of-multiplicities functions  $${\mathcal{U}}_k = \sum_{1 \leq a_1 < a_2 \dots \leq a_k} \frac{q^{a_1 + \dots + a_k}}{(1-q^{a_1})^2 \dots (1-q^{a_k})^2}.$$

In (\cite{KeithRamanujan}, \cite{KeithIJNT}), the second author studied $\nu_2$ and $\nu_3$ and proved properties including several congruences, among which the example of interest in this paper is \begin{theorem}$$\nu_2(16n+14) \equiv 0 \pmod{4}.$$ \end{theorem}

That $\nu_2(16n+14) \equiv 0 \pmod{2}$ follows immediately from simple combinatorial argument: a self-conjugate partition into two part sizes must be a difference of squares, which cannot be $2 \pmod{4}$.  The proof of the congruence modulo 4, in contrast, required the use of modular forms.  A combinatorial proof seemed that it should be possible and desirable.  Such a proof is provided here.  In the process, we make several refinements of parity to subsets of the partitions being considered.

The next section includes sufficient background to make the paper self-contained.  The following section then produces the combinatorial proof of the theorem.  Section \ref{Dn} contains certain results of interest on the sizes of individual subclasses which were tangential to the main proof, and we conclude in the final section with a few additional remarks.

\section{Background}

The \emph{Ferrers diagram} for a partition is the unique top-left justified grid of boxes where each part corresponds to a row of boxes. For instance, the Ferrers diagram for the partition $(5,4,1)$ is
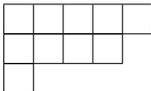
\[\young(~~~~~,~~~~,~)\]
The involution of \emph{conjugation} on partitions is reflection across the main diagonal of the Ferrers diagram.  We say two partitions are \textit{conjugates} of one another if they differ by a reflection across the main diagonal. For instance, we say that $(5,4,1)$ and $(3,2,2,2,1)$ are conjugate. A partition is called \textit{self-conjugate} if it is a fixed point of this operation, i.e. if its Ferrers diagram is symmetric across the diagonal. For instance, $(4,2,1,1)$ with Ferrers diagram \[\young(~~~~,~~,~,~)\] is self-conjugate.

In the literature, the conjugate of the partition $\lambda$ is normally denoted by $\overline{\lambda}$ or $\lambda^\prime$.  Since we will need to employ conjugation as one of a sequence of bijections, we will write this operation as the function $Conj(\lambda)$.

For any partition of $n$, let $\lambda_i$ denote the $i$-th largest part size and let $m_i$ be its multiplicity. Then for any partition counted by $\nu_k(n)$, we may write $n=\sum_{i=1}^k\lambda_im_i$. It will be helpful to write partitions using the frequency notation $(\lambda_1^{m_1}\lambda_2^{m_2}\cdots \lambda_k^{m_k})$. For instance, the partitions counted by $\nu_2(6)$ are written
$$(5^11^1),\quad (4^12^1),\quad (4^11^2),\quad (3^11^3),\quad (2^21^2),\quad (2^11^4).$$

It may now be immediately observed that if a partition into two part sizes is self-conjugate, its Ferrers diagram is literally a geometric difference of squares, i.e. a square with a square removed from its lower right corner, and hence $\lambda$ must partition a difference of squares, $x^2 - y^2$; but since $x^2 - y^2$ is never $2 \pmod{4}$, no partition of $4n+2$ into 2 part sizes can be self-conjugate, and so conjugation is a matching without fixed points on this set.  Hence $\nu_2(4n+2) \equiv 0 \pmod{2}$.

We can write an explicit formula for the conjugate of a partition into exactly two part sizes:

\begin{lemma}\label{conjlemma} \begin{equation}
    \textrm{Conj}(\lambda_1^{m_1}\lambda_2^{m_2})=(m_1 + m_2)^{(\lambda_2)}(m_1)^{(\lambda_1-\lambda_2)}.
    \label{conj}
\end{equation}
\end{lemma}

Write any positive integer $m$ uniquely as $m=2^k t$ with $t$ odd. We will denote by $k=b(m)$ the 2-adic valuation of $m$ and $t=\ell(m)$ the largest odd divisor of $m$.

\section{Combinatorial proof of the main theorem}

The strategy of the proof is as follows.  We begin by establishing several structural properties of the set of partitions under consideration.  We next define involutions which, thanks to the relevant properties, will lack fixed points, showing an exhaustive collection of subsets each contain an even number of conjugate pairs.

We first distinguish several subsets by the parities of their sizes and their respective multiplicities.  Let $(\lambda_1^{m_1}\lambda_2^{m_2})$ be called a member of the \emph{parity class} $ABCD$ in which $A, B, C, D \in \{O, E\}$ depending on whether $\lambda_1, m_1,\lambda_2, m_2$ respectively are odd or even. For instance, consider the partition $(6^14^2)$. Since the elements of the list $(6,1,4,2)$ are respectively even, odd, even, and even, we have $(6^14^2)\in EOEE$.  

The nonempty parity classes are enumerated in the following lemma.

\begin{lemma}\label{classes} Partitions of $4n+2$ into two part sizes can only be of types $OOOO$, $EOOE$, $EEOE$, $EOEO$, $OEEE$, $EEEO$, $EOEE$, $OEOE$, and $OEEO$.  Furthermore, for partitions $(\lambda_1^{m_1} \lambda_2^{m_2})$ not in the class $OOOO$, exactly one of the two products has $\lambda_i m_i \equiv 2 \pmod{4}$, and the other has $\lambda_j m_j \equiv 0 \pmod{4}$. 
\end{lemma}

\begin{proof} The seven excluded classes are $EEEE$, or those in which one pair of symbols is $OO$ and the other not.  It is not possible for one pair to contribute an odd amount to the sum, and the other an even amount, nor for both pairs to contribute $0 \pmod{4}$.

In $OOOO$, each pair contributes $1 \pmod{2}$.  All remaining parity classes have each pair contribute an even amount, but these cannot be both $0 \pmod{4}$, nor both $2 \pmod{4}$, either of which would sum to $0 \pmod{4}$.
\end{proof}

An immediate consequence of Lemma \ref{conjlemma} is that these classes conjugate between themselves, i.e. elements of one class always conjugate to another, except for the class $OEEO$, which is self-conjugate - the partitions as a class, not individually.

\begin{table}[h]
    \begin{tabular}{c  c}
        $EOOE$ & $OOOO$ \\ \hline
        $EEOE$  & $EOEO$ \\ \hline
        $OEEE$  & $EEEO$ \\ \hline
        $EOEE$  & $OEOE$ \\ \hline
        $OEEO$ & self\\
    \end{tabular}
    \vskip.1in
    \caption{The nonempty parity classes for partitions counted by $\nu_2(4n+2)$, paired by conjugation.}
    \label{parities}
\end{table}

Per Lemma \ref{classes}, if a partition is not in $OOOO$, we use an overline to denote which part-multiplicity pair contributes $2\pmod 4$. For instance, $(6^11^8)\in \overline{EO}OE$ and $EOOE=\overline{EO}OE \cup EO\overline{OE}$.

\begin{lemma}\label{oeeolemma} We have that $|OEEO|$ is always even. Moreover, $|\overline{OE}EO|=|OE\overline{EO}|$. 
\end{lemma}

\begin{proof} No partition under consideration is self-conjugate but the conjugates of a partition in $OEEO$ are also in $OEEO$, hence $|OEEO|$ is even.  Lemma \ref{conjlemma} reveals that conjugation will swap $m_1$ and $\lambda_2$.  Since the other elements of the partition are odd, exactly one of $m_1$ and $\lambda_2$ is divisible by 4 and the other is $2 \pmod{4}$.  Hence $|\overline{OE}EO|=|OE\overline{EO}|$.
\end{proof}

In order to avoid visual clutter, for the remainder of the paper we suppress the use of $| \cdot |$ for set size and $( \cdot )$ for writing partitions.

\begin{lemma}\label{odd pairs}
    If a partition $\lambda_1^{m_1}\lambda_2^{m_2}$ is counted by $\nu_2(16n+14)$, then at least one of the part-multiplicity pairs has $\ell(\lambda_i)\not\equiv\ell(m_i)\pmod 8$.
\end{lemma}

\begin{proof}
    For sake of contradiction, assume $\ell(\lambda_1)\equiv\ell(m_1)$ and $\ell(\lambda_2)\equiv\ell(m_2)\pmod 8$.
    Recall that odd squares are $1\pmod 8$.
    If $\lambda_1^{m_1}\lambda_2^{m_2}\in OOOO$, then
    \[n=\lambda_1m_1+\lambda_2m_2=\ell(\lambda_1)\ell(m_1) + \ell(\lambda_2)\ell(m_2)\equiv 2\pmod 8.\]
    Otherwise, $\lambda_1m_1$ and $\lambda_2m_2$ are both even. Dividing by 2 reveals
    \begin{alignat*}{2}
        \frac{n}{2} &= 2^{b(\lambda_1)+b(m_1)-1}\ell(\lambda_1)\ell(m_1) + 2^{b(\lambda_2)+b(m_2)-1}\ell(\lambda_2)\ell(m_2)\\
        &\equiv 2^{b(\lambda_1)+b(m_1)-1} + 2^{b(\lambda_2)+b(m_2)-1} \pmod{8}.
    \end{alignat*}
    Recall the powers of two are $0,1,2,$ and $4\pmod 8$. Observe that $\frac{n}{2}\equiv 7\pmod 8$ cannot be written as a sum of two powers of two. Contradiction.
\end{proof}

We now begin defining our maps.

\begin{definition} Let $\lambda$ be a partition such that $2^{b(\lambda_1)} \ell (m_1) \neq  2^{b(\lambda_2)} \ell (m_2)$.  Define $\rho$ to be the map that swaps the largest odd divisor of each part with that of its multiplicity, and reorders the resulting parts if necessary.  That is,
\begin{multline*}  \rho (\lambda_1^{m_1}\lambda_2^{m_2}) = \rho \left( (2^{b(\lambda_1)} \ell(\lambda_1))^{(2^{b(m_1)} \ell(m_1))} (2^{b(\lambda_2)} \ell(\lambda_2))^{(2^{b(m_2)} \ell(m_2))} \right) \\ 
=     \begin{cases} 
      (2^{b(\lambda_1)} \ell(m_1))^{(2^{b(m_1)} \ell(\lambda_1))} (2^{b(\lambda_2)} \ell(m_2))^{(2^{b(m_2)} \ell(\lambda_2))} & 2^{b(\lambda_1)} \ell(m_1) > 2^{b(\lambda_2)} \ell(m_2) \\
      (2^{b(\lambda_2)} \ell(m_2))^{(2^{b(m_2)} \ell(\lambda_2))} (2^{b(\lambda_1)} \ell(m_1))^{(2^{b(m_1)} \ell(\lambda_1))}  & 2^{b(\lambda_1)} \ell(m_1) < 2^{b(\lambda_2)} \ell(m_2). \\
   \end{cases}
\end{multline*}
\end{definition}

For instance,
    \[\rho(4^63^2)=\rho\left((2^2*1)^{(2^1*3)}(2^0*3)^{(2^1*1)}\right)=(2^2*3)^{(2^1*1)}(2^0*1)^{(2^1*3)}=12^21^6.\]
    
This map gives us evenness of a collection of parity classes.

\begin{corollary}\label{even parities} If $n\equiv 14\pmod {16}$, then $$EOOE+EEOE+OEEE+OEEO\equiv 0\pmod 2.$$
\end{corollary}

\begin{proof} By Table \ref{parities}, the parity classes listed are those with $\lambda_1\not\equiv\lambda_2\pmod 2$.  Hence $2^{b(\lambda_1)} \neq 2^{b(\lambda_2)}$, and $\rho$ is well-defined on these classes.

If $\rho$ preserved the parity of $\lambda_1$ and $\lambda_2$, the part sizes did not swap.  Then by Lemma \ref{odd pairs}, one of the part sizes changes, so $\lambda$ is not fixed by $\rho$. Otherwise, $\rho$ swapped the parities of $\lambda_1$ and $\lambda_2$ and changed the parity class of the input partition.\\

Hence, $\rho$ has no fixed points in these parity classes.
\end{proof}

This now gives us 

\begin{lemma}\label{sixgroup}
    If $n\equiv 14\pmod {16}$, then $$EOOE+EEOE+OEEE+OOOO+EOEO+EEEO\equiv 0\pmod 4.$$
\end{lemma}

\begin{proof}
    Recall that $OEEO$ is even. Combining with Corollary \ref{even parities} shows that $EOOE+EEOE+OEEE$ is even. The other classes listed in the lemma are the conjugates of these three classes, and hence are of equal size.  The total size of all six classes is thus divisible by 4.
\end{proof}

We now handle the remaining classes by proving the following lemma.

\begin{lemma}
    If $n\equiv 6\pmod 8$, then $$EOEE+OEOE+OEEO \equiv 0\pmod 4.$$
\end{lemma}

\begin{proof}
    By conjugation, $OEOE=EOEE$. Hence it suffices to show that $OEOE+\frac{1}{2}OEEO$ is even. Since by conjugation, $\overline{OE}EO=OE\overline{EO}$, we aim to show that $OEOE + \overline{OE}EO$ is even.

    For the classes other than $OOOO$, define the map $\overline{\phi}$ to swap the 2-adic valuations of the part and multiplicity contributing $2\pmod 4$, and reorder the resulting parts if necessary. For instance,
    \[\overline{\phi}(3^21^{8})=\overline{\phi}\left((2^0*3)^{(2^1*1)}(2^0*1)^{(2^3*1)}\right)=(2^1*3)^{(2^0*1)}(2^0*1)^{(2^3*1)}=6^11^8.\]

    Observe the following equivalence for $OEOE$. In the final step, we apply $\overline{\phi}$ to each set.
    \begin{alignat*}{2}
       OEOE &= \overline{OE}OE + OE\overline{OE} \\ 
        &=\overline{OE}OE+ \{OE\overline{OE}:\lambda_1<2\lambda_2\}+ \{OE\overline{OE}:\lambda_1>2\lambda_2\}\\
        &= \{\overline{EO}OE: \lambda_1>2\lambda_2\} + \{\overline{EO}OE:\lambda_1<2\lambda_2\} + OE\overline{EO}.\\
    \end{alignat*}
    Since $OEEO$ is even, the count of partitions under consideration
    \begin{alignat*}{2}
        OEOE + \overline{OE}EO &= \{\overline{EO}OE: \lambda_1>2\lambda_2\} + \{\overline{EO}OE:\lambda_1<2\lambda_2\} + OE\overline{EO} + \overline{OE}EO\\
        &\equiv \{\overline{EO}OE: \lambda_1>2\lambda_2\} + \{\overline{EO}OE:\lambda_1<2\lambda_2\}\pmod 2.
    \end{alignat*}
    
    We now show that both of these subsets, individually, are of even cardinality. We require the following lemma showing that $\rho$ is well defined on the subset of $OOOO$ we need.

\begin{lemma} The conjugates of partitions in $\overline{EO}OE$ wherein $\lambda_1 \neq 2\lambda_2$ are those partitions in $OOOO$ where $m_1 \neq m_2$.
\end{lemma}

\begin{proof} The conjugate of a partition $\lambda_1^m \lambda_2^m \in OOOO$ is $(2m)^{\lambda_1 + \lambda_2} m^{\lambda_1}$.
\end{proof}

In fact, we can pair these partitions by first conjugating from $EOOE\rightarrow OOOO$, swapping all parts and multiplicities, then conjugating back from $OOOO\rightarrow EOOE$. 
    \[\overline{EO}OE \, \, {\phantom{.}}^{\underrightarrow{\textrm{Conj}}} \, \, OOOO \, \, {\phantom{.}}^{\underrightarrow{\rho}} \, \, OOOO \, \, {\phantom{.}}^{\underrightarrow{\textrm{Conj}}} \, \, \overline{EO}OE \]

    By Lemma \ref{conjlemma}, these maps can be explicitly written for partitions in $\overline{EO}OE$ such that $\lambda_1\neq 2\lambda_2$ by

    \[ (\textrm{Conj}\circ\rho\circ\textrm{Conj})(\lambda_1^{m_1}\lambda_2^{m_2})=
    \begin{cases} 
      (2m_1+m_2)^{\lambda_2}m_1^{(\lambda_1-2\lambda_2)} & \lambda_1 > 2\lambda_2 \\
      (2m_1+m_2)^{(\lambda_1-\lambda_2)}(m_1+m_2)^{(2\lambda_2-\lambda_1)} & \lambda_1 < 2\lambda_2.
   \end{cases}
\]

Thus $\textrm{Conj}\circ\rho\circ\textrm{Conj}$ preserves each set.  We now must confirm that it has no fixed points.  Since conjugation is a bijection which maps partitions from one parity class to another, it suffices to show that $\rho$ has no fixed points in $OOOO$ among the conjugates of partitions in $\{\lambda \in \overline{EO}OE: \lambda_1 \neq 2\lambda_2\}$.  (In general, fixed points such as $7^11^7$ do exist, even among partitions of $16n+14$.)

\begin{lemma} No conjugate of a partition in $\{\lambda \in \overline{EO}OE: \lambda_1 \neq 2\lambda_2\}$ is a fixed point of $\rho$.
\end{lemma}

\begin{proof}For sake of contradiction, suppose a partition in $\overline{EO}OE$ conjugated to a fixed point of $\rho$ in $OOOO$. Such a fixed point must look like $a^a b^b$ or $a^b b^a$, with $a$ and $b$ odd.

Since $n\equiv 6\pmod 8$ cannot be written as a sum of two odd squares, the problematic partition must have the latter form. Then by Lemma \ref{conjlemma},

    \[\textrm{Conj}(a^bb^a)=(a+b)^bb^{(a-b)}.\]

    Since $\textrm{Conj}(a^bb^a)\in\overline{EO}OE$, it follows that $a+b\equiv 2\pmod 4$. Having both $a,b$ odd requires that $a\equiv b\pmod 4$. However, this implies that $n=2ab\equiv 2\pmod 8$. Contradiction.
\end{proof}

Having shown that $OEOE + \overline{OE}EO$ is even, and knowing that $OEOE = EOEE$ by conjugation and likewise $\overline{OE}EO = OE\overline{EO}$, we have that $EOEE + OEOE + OEEO \equiv 0 \pmod{4},$ and the theorem is proved.
\end{proof}

\section{Subclass sizes}\label{Dn}

Andrews \cite{GEA1} and MacMahon \cite{MacMahon} both derived a compact formula for the value of $\nu_2(n)$, finding $$\nu_2(n) = \frac{1}{2} \left( \sum_{k=1}^{n-1} d(k)d(n-k) - \sigma_1(n) + d(n) \right)$$ where $d(n)$ and $\sigma_1(n)$ are the divisor function and sum of divisors function, respectively.  Since $8j+7$ is not a square, it is immediate that $d(16j+14)$ is divisible by 4, and that $\sigma_1(16j+14)$ is divisible by 8.

Lemma \ref{sixgroup} tells us that six of the subclasses together have a total cardinality divisible by 4, and hence the truth of the final result is equivalent to the same being true for the other three groups, $EOEE + OEOE + OEEO$.  We have the following fact, which we prove with another map.

\begin{lemma}\label{twiceOEOE} If $n \equiv 2 \pmod{4}$, then $$EOEO = OEOE + \frac{\sigma_1(n/2)}{2} - \frac{d(n)}{4}.$$
\end{lemma}

\begin{proof} Let $\tau: OEOE \rightarrow EOEO$ be the map that swaps the 2-adic valuations of each part-multiplicity pair, and exchanges the two part sizes if necessary.  Note that the map is always well-defined in this direction.  The complement of its image is exactly those partitions in $EOEO$ where $\ell(\lambda_1) = \ell(\lambda_2) =: k$.  These partitions are of the form $$(2^{1+t}k)^a (2k)^b \quad , \quad t > 0 \quad , \quad k,a,b \equiv 1 \pmod{2} .$$ Hence they are in bijection with partitions of $n/2$ of the form $$(2^t k)^a (k)^b\quad , \quad t > 0 \quad , \quad k,a,b \equiv 1 \pmod{2} .$$  Hence $$\frac{n}{2} = k (2^t a + b).$$  But $n/2$ is odd, as is $k$, and $2^t a$ with $a$ odd is a unique way to write some even number.  Hence every positive solution of $$\frac{n}{2} = k (x+y)$$ of this equation yields such a partition, except that we take exactly the half of the solution set in which $x$ is even.  For each divisor $k$ there are $\frac{1}{2} \left(\frac{n}{2k} - 1\right)$ associated solutions, and the claim follows.
\end{proof}

Hence for $n \equiv 2 \pmod{4}$, it holds that $$2(OEOE) = 2(EOEO) - \sigma_1(n/2) + \frac{d(n)}{2}.$$  By Lemma \ref{odd pairs}, if $n \equiv 14 \pmod{16}$, the map $\rho$ has no fixed points in $EOEO$, and is an involution on this set, so $2(EOEO) \equiv 0 \pmod{4}$, and we know $\sigma_1(8j+7) \equiv 0 \pmod{8}$. Hence, after conjugating one copy of $OEOE$, we have the following result.

\begin{theorem}\label{twogroup} If $n \equiv 14 \pmod{16}$, then $$EOEE + OEOE \equiv \frac{d(n)}{2} \pmod{4}.$$
\end{theorem}

But as we commented at the beginning of this section, $d(16j+14) \equiv 0 \pmod{4}$ and so $d(n)/2$ is even.  Hence truth of $$EOEE + OEOE + OEEO \equiv 0 \pmod{4}$$ yields the final corollary of this section,

\begin{corollary}\label{OEEOvalue} If $n \equiv 14 \pmod{16}$, then $OEEO \equiv \frac{d(n)}{2} \pmod{4}$.
\end{corollary}

Conversely, if Corollary \ref{OEEOvalue} could be proved independently, the main theorem of the paper would also follow.

\section{Future investigations}

The proof given here does not immediately generalize to other known congruences for $\nu_2$, such as $\nu_2(36n+30) \equiv 0 \pmod{4}$, so a natural question would be whether it can be adapted for these classes.  Indeed, it may be possible to prove an infinite class of congruences of this type by isolating what features of a numerical progression suffice to generalize the arguments.

The maps produced in this paper can apply more generally to larger classes of partitions.  It could be interesting to determine their utility in the proof of other identities.  The map $\overline{\phi}$ is of limited applicability, but the map $\rho$ which swaps the largest odd divisor of each part size and its multiplicity is an involution on the set of partitions wherein the values $2^{\lambda_i} \ell(m_i)$ are distinct.

Conjugation is the most elementary means of proving evenness of sets of partitions; when proving divisibility by 4, it would seem at first glance that pairing up conjugate pairs is a natural strategy to attempt, but to our knowledge this is not a common method of proof in combinatorial partition theory.  Of course, this may simply be due to the relative scarcity of combinatorially provable congruences modulo 4, but perhaps exploration of the technique could yield new results.

We wish to raise an alternative point of interest: another common method of proving that a set of partitions has a given divisibility, say by $m$, is to establish a \emph{rank statistic} on the set being considered, which divides the set into $m$ equal classes, or into classes each of which itself has cardinality divisible by $m$.  The usual rank of a partition is \emph{Dyson's rank}, after its proposer Freeman Dyson, who conjectured \cite{Dyson} that $$rk(\lambda) := \lambda_1 - \sum m_i,$$ that is, the largest part of a partition minus its number of parts, divides the partitions of $5n+4$ into five equal classes when reduced modulo 5.  This conjecture was later proved by Atkin and Swinnerton-Dyer \cite{ASD}.

Calculation has not suggested an equidistribution statistic for the congruences known to exist for $\nu_2(n)$.  However, we have observed data which suggests the following facts which would imply congruences modulo 4:

\begin{conjecture} For $\lambda = (\lambda_1^{m_1}, \lambda_2^{m_2})$, let $rk(\lambda)$ be Dyson's rank and let $rk_2(\lambda) = \lambda_1 + \lambda_2 - 2m_1 - m_2$.  Then it holds for $(A,B) \in \{(16,14), (36,30), (72,42), (196,70), (252,114) \}$ that the numbers of partitions of $An+B$ into two part sizes with $rk_2(\lambda)$ odd, or even, are both divisible by 4.  For Dyson's rank the same claim holds except for $16n+14$.
\end{conjecture}

In fact, in both cases the numbers of partitions with the statistic being $0 \pmod{4}$ are particularly divisible by 4, although this is superfluous to proof of the congruence.  This is the entire list of congruences current known for $\nu_2(An+B)$, and one wonders if the property holds in general for all such congruences.  A further interesting observation is that a more refined statistic, the \emph{crank}, which witnesses considerably more congruences for ordinary partitions than the rank and is generally considered the more useful and deeper statistic, does \emph{not} have the same divisibility properties for these sequences.

\end{document}